\documentclass[11pt]{amsart}

\usepackage[margin=2.5cm]{geometry}
\usepackage[colorlinks=true, linkcolor=blue, citecolor=blue]{hyperref}
\usepackage{graphicx}
\usepackage{framed}
\usepackage{amsmath}
\usepackage{amssymb}
\usepackage{amsfonts}
\usepackage{mathrsfs}
\usepackage{array}
\usepackage{enumerate}
\usepackage[]{enumitem}
\usepackage{mathtools}
\usepackage{tikz-cd}
\usepackage{multirow}
\usepackage{float}

\usepackage{amsthm}  

\usepackage{cite}
\usepackage{url}

\newcommand{\QQ}{\mathbb{Q}}
\newcommand{\ZZ}{\mathbb{Z}}

\newcommand{\PP}{\mathbb{P}}
\newcommand{\FF}{\mathbb{F}}

\newcommand{\map}{\rightarrow}

\theoremstyle{plain} 
\newtheorem{lemma}{Lemma}[section]
\newtheorem{theorem}[lemma]{Theorem}
\newtheorem{proposition}[lemma]{Proposition}

\newtheorem{claim*}{Claim}

\theoremstyle{remark}
\newtheorem{remark}[lemma]{Remark}

\begin{document}
\title{Rational points on K3 surfaces of degree 2}
\author{Júlia Martínez-Marín}
\address{School of Mathematics, University of Bristol, Bristol, BS8 1UG, UK}
\email{julia.martinezmarin@bristol.ac.uk}

\begin{abstract}
    A K3 surface over a number field has infinitely many rational points over a finite field extension. For K3 surfaces of degree 2, arising as double covers of $\PP^2$ branched along a smooth sextic curve, we give a bound for the degree of such an extension. Moreover, using ideas of van Luijk and a surface constructed by Elsenhans and Jahnel, we give an explicit family of K3 surfaces of degree 2 defined over $\QQ$ with geometric Picard number 1 and infinitely many $\QQ$-rational points that is Zariski dense in the moduli space of K3 surfaces of degree 2.
\end{abstract}

\maketitle

\section{Introduction}

A K3 surface over an algebraically closed field contains a rational (possibly singular) curve \cite[Appendix]{mm}. Therefore a K3 surface over a number field $K$ has infinitely many rational points over a finite field extension of $K$ over which the rational curve is defined and has a rational point. For a K3 surface $X$ of degree 2, i.e. a double cover of $\PP^2$ branched along a smooth sextic curve, we find a bound (independent of $X$ and $K$) for the degree of an extension $L$ of $K$ such that $X$ has infinitely many $L$-rational points. It remains open whether the bound is sharp or not. 

\begin{theorem}\label{bound}
     Let $X$ be a K3 surface of degree 2 over a number field $K$. Then there is a field extension $L/K$ with $[L:K]\leq 12$ such that $X(L)$ is infinite.
\end{theorem}

We may then wonder whether the set of rational points over a finite field extension of $K$ is Zariski dense in $X$. If so, we say that the rational points on $X$ are \emph{potentially dense}. Apart from K3 surfaces, all other surfaces of Kodaira dimension $\kappa=0$ have potentially dense rational points \cite[Proposition 4.2]{hasset}, \cite[Theorem 3.1]{enriques}. For K3 surfaces, it is conjectured that the rational points are also potentially dense. In \cite{bogo2}, Bogomolov and Tschinkel show that any double cover of $\PP^2$ branched along a sextic curve with a singular point has potentially dense rational points. The blow-up in the preimage of the singular point under the double cover is an elliptic K3 surface. More generally, Bogomolov and Tschinkel proved potential density of rational points for all elliptic K3 surfaces \cite{Bog}. Such surfaces have geometric Picard number $\rho \geq 2$. 

In \cite{vluijk}, van Luijk gave the first explicit example of a K3 surface over $\QQ$ with geometric Picard number $\rho=1$. With this K3 surface, which is of degree 4, he also gave a positive answer to the question of whether there exists a K3 surface over a number field with infinitely many rational points and geometric Picard number 1, asked by Swinnerton-Dyer and Poonen at a workshop in 2002. Using this, he proved the following theorem. 
\begin{theorem}[van Luijk, 2005 \cite{vluijk}]
In the moduli space of K3 surfaces of degree 4, the set of surfaces defined over $\QQ$ with geometric Picard number 1 and infinitely many $\QQ$-rational points is Zariski dense.     
\end{theorem}

Using van Luijk's method, Elsenhans and Jahnel constructed many examples of K3 surfaces of degree 2 with geometric Picard number 1 in \cite{eljah}. Modifying slightly one of their examples and using the method developed in the proof of Theorem \ref{bound} to find infinitely many rational points, we prove the analogue of van Luijk's theorem for K3 surfaces of degree 2.

\begin{theorem}\label{moduli}
In the (coarse) moduli space of K3 surfaces of degree 2, the set of surfaces defined over $\QQ$ with geometric Picard number 1 and infinitely many $\QQ$-rational points is Zariski dense. 
\end{theorem}

The proof of Theorem \ref{bound} is divided into two steps. The first consists of finding a curve of geometric genus 1 lying on the K3 surface. This can always be done up to enlarging the base field by an extension of degree at most 6, as we prove in Lemma \ref{deg6ext}. The proof of this lemma is constructive, so given an explicit K3 surface $X$ of degree 2 over a number field, one can find an explicit extension over which $X$ contains a curve of geometric genus 1. In fact, in most cases it will suffice to take the smallest extension $L$ of $K$ over which the branch curve $B$ of $X\map \PP^2$ has $B(L)\neq \emptyset$. This gives a way of finding examples of K3 surfaces (of degree 2) with infinitely many rational points over their base field, and this is the key to proving Theorem \ref{moduli}. See Remark \ref{rem:constructive} for more details on this. 

The second ingredient to prove Theorem \ref{bound} involves showing that there is a degree 2 extension $L$ of a number field $K$ over which a smooth curve $C$ of genus 1 given by $y^2=h(x)$ for a polynomial $h\in K[x]$ of degree 4 has infinitely many $L$-rational points. It is not hard to see that this happens if we allow an extension of degree 4 since we can find a quadratic extension $K'$ of $K$ such that $C(K')\neq \emptyset$, so $C$ becomes an elliptic curve over $K'$. Then we can use Merel's theorem \cite[p.~242]{AEC} to see that $C_{K'}$ must have a non-torsion point over some quadratic extension $K''$ of $K'$, i.e. an extension of $K$ of degree 4. In Proposition \ref{quadraticext} we do both steps at once and show that we only require a degree 2 extension.

\subsection{Outline}
We begin in Section \ref{preliminaries} with a brief review on K3 surfaces and relevant invariants of algebraic surfaces. In Section \ref{x6y6z6}, we discuss an example of a K3 surface of degree 2 with no rational points over $\QQ$ but Zariski dense rational points over a quadratic field extension. 
In Section \ref{thm1}, we prove Lemma \ref{deg6ext}, which says that a K3 surface of degree 2 over a number field $K$ contains a singular curve of geometric genus 1 over an extension $K'$ of $K$ of degree at most 6. We show that the normalisation $C$ of this curve is given by $y^2=h(x)$, for a separable polynomial $h\in K'[x]$ of degree 4. Proposition \ref{quadraticext} now says that $C$ becomes an elliptic curve with a non-torsion point over infinitely many quadratic extensions. These two results are combined to prove Theorem \ref{bound}. In Section \ref{thm2}, we give an explicit family of K3 surfaces of degree 2 defined over $\QQ$ with geometric Picard number 1 and infinitely many rational points. Theorem \ref{moduli} follows once we show that this family is Zariski dense in the moduli space of K3 surfaces of degree 2.

\subsection*{Acknowledgements} Most of the results in this paper first appeared in my Master's thesis and I am very grateful to my advisor, Giacomo Mezzedimi, for suggesting the topic and for all the help, time and encouragement he gave me, as well as his valuable comments on an earlier version of the material presented here. I would also like to thank Ronald van Luijk and Matthias Schütt for their suggestions regarding the proof of Proposition \ref{denseratpts}, and Sam Frengley for his insightful comments and his assistance with \textsc{Magma}. I also wish to thank the referee for their helpful suggestions, which have improved the exposition of this paper. This paper was prepared during my time as a PhD student at the University of Bristol, where I am supported by the Heilbronn Institute for Mathematical Research. 

\section{Preliminaries} \label{preliminaries}
Let $X$ be a smooth, projective surface over a field $k$. 
The \emph{Néron--Severi group} of $X$ is the quotient \(\operatorname{NS}X\coloneq\operatorname{Pic} X/\operatorname{Pic}^0X,\)
where $\operatorname{Pic}^0X$ denotes the connected component of the Picard variety $\operatorname{Pic}X$ containing the identity element.
The Néron--Severi group $\operatorname{NS}X$ is a finitely generated abelian group \cite[p.~7]{Huyb}. 
The \emph{Picard number} of $X$ is $\operatorname{rk} \left(\operatorname{NS}X\right)$ and its \emph{geometric Picard number} is $\operatorname{rk} \left(\operatorname{NS} X_{\overline{k}}\right)$, where $X_{\overline{k}}=X\times_k \overline{k}$.
A \emph{genus 1 fibration} is a morphism $\varphi:X\rightarrow C$ onto a smooth, projective, irreducible curve $C$ with generic fibre a smooth curve of genus 1. When $\varphi$ has a section, we call it an \emph{elliptic fibration}.

A \emph{K3 surface} is a smooth, projective, geometrically integral surface $X$ with trivial canonical bundle and $H^1(X,\mathcal{O}_X)=0$. If $X$ is a K3 surface, then $\operatorname{Pic}^0X$ is trivial, $\operatorname{NS}X\cong\operatorname{Pic}X$ is torsion free \cite[Proposition VIII.3.2]{Barth}, and its geometric Picard number is at most 22 (in characteristic 0, it is at most 20) \cite[p.~13]{Huyb}. A \emph{polarised K3 surface of degree} $d>0$ consists of a K3 surface $X$ together with an ample line bundle $\mathcal{L}$ such that $\mathcal{L}$ is primitive, i.e. indivisible in $\operatorname{Pic}X$, with $\mathcal{L}^2=d$ \cite[p.~31]{Huyb}.
By the Riemann--Roch formula the degree of a K3 surface is always even. For any even number $d>0$ there exists a K3 surface of degree $d$ \cite[p.~32]{Huyb}. K3 surfaces of degree 2 arise as double covers of $\PP^2$ branched along a smooth sextic curve, so they are given in the weighted projective space $\PP(1,1,1,3)=\operatorname{Proj}\QQ [x,y,z,w]$ by equations of the form $w^2=f(x,y,z)$, where $f$ is a homogeneous polynomial of degree 6 such that the curve $V(f)\subset \PP^2$ given by the vanishing of $f$ is smooth.

\section{An example} \label{x6y6z6}
Consider the hypersurface $X$ in $\PP(1,1,1,3)$ given by
\[
w^2=-x^6-y^6-z^6.
\]
This is a double cover of $\PP^2_\QQ$ branched along the sextic curve $B=V(-x^6-y^6-z^6)$. Using the Jacobian criterion, one can check that $B$ is smooth, so $X$ is a K3 surface of degree 2. Notice that $X$ has no rational points over $\QQ$. However, when base-changing to the quadratic extension $\QQ(i)$, $X$ already attains at least 3 points: $[1:0:0:i]$, $[0:1:0:i]$ and $[0:0:1:i]$. In fact, $X(\QQ(i))$ is infinite, since $X_{\QQ(i)}$ contains, for example, the line $V(w-ix^3,y-iz)$. It turns out that the set of $\QQ(i)$-rational points $X(\QQ(i))$ is Zariski dense in $X$. This follows from Proposition \ref{denseratpts}, since $X$ is isomorphic over $\QQ(i)$ to the surface $Y$ in the proposition.

\begin{proposition} \label{denseratpts}
Let $Y$ be the hypersurface in $\PP(1,1,1,3)$ given by 
\[
w^2=x^6+ y^6-z^6. 
\]   
Then the set of rational points $Y(\QQ)$ is Zariski dense in $Y$. 
\end{proposition}

\begin{proof}
To show this we use \cite[Proposition 3.8]{twofibs}, which states that if a smooth projective surface $Y$ admits two different genus 1 fibrations $f_1$ and $f_2$, then there exists an explicit Zariski closed subset $Z\subset Y$ such that if $Y(\QQ)\backslash Z(\QQ)\neq \emptyset$ then $Y(\QQ)$ is Zariski dense in $Y$. 
In our case, it is enough to take $Z$ to be the union of the singular fibres of $f_1$ and $f_2$ and the sets $T_1$ and $T_2$ which we now define. First, for $i=1,2$ suppose we have a rational map $\alpha_i:Y\dashrightarrow Y$ respecting $f_i$ such that its restriction to the generic fibre $\mathcal{Y}_i$ is not merely a translation by an element of the Jacobian $J(\mathcal{Y}
_i)$. We define $T_i$ to be the union of all points $P$ lying on smooth fibres of $f_i$ such that $\alpha_i(P)$ is a torsion point on the smooth fibre $f_i^{-1}(f_i(P))$ with distinguished point $P$. For more details, see \cite[Section 2]{twofibs}. 

Let us therefore consider the following maps:
\begin{align*}
f_1:  \ \ Y  &  \xrightarrow{} \ \PP^1\\
[x:y:z:w] \ &\mapsto \ [w+x^3 : y^3-z^3],
\end{align*}
\begin{align*}
f_2:  \ \ Y  &  \xrightarrow{} \ \PP^1\\
[x:y:z:w] \ &\mapsto \ [w+x^3 : y^3+z^3].
\end{align*}
Both maps are morphisms outside pure codimension 1 subschemes of $Y$, and since rational maps extend along codimension 1 \cite[Proposition 4.1.16]{Liu}, both $f_1$ and $f_2$ extend to morphisms from $Y$ to $\PP^1$. One can check that $f_1$ and $f_2$ are genus 1 fibrations by using a similar argument to that in \cite[Proof of Theorem 2.1]{Cornnaka}.

Consider the lines $M_1=V(w-x^3, y+z)$ and $M_2=V(w-x^3, y-z)$ contained in $Y$. For each $i=1,2$, note that $M_i$ is a multisection of $f_i$ of degree 3, i.e., $M_i.f_i^{-1}(P)=3$ for general $P\in \PP^1$. Following \cite[Remark 2.4]{twofibs}, we consider the rational map $\alpha_i:Y \dashrightarrow Y $ defined on smooth fibres of $f_i$ by $\alpha_i(P)=R$, where $R$ is the unique point on the fibre $F=f_i^{-1}(f_i(P))$ for which $\mathcal{O}_F(R)$ is isomorphic to the degree-1 line bundle $\mathcal{O}_F(M_i)\otimes\mathcal{O}_F(-2P)$.

Consider the point $P = [2: 3/2:1: 69/8]\in Y(\QQ)$. Then, $P$ is contained in the fibre $F_1=f_1^{-1}([7:1])$, given by
\[
7(w-x^3) = y^3+z^3,
\]
and in the fibre $F_2 := f_2^{-1}([19/5:1])$, given by
\[
19(w-x^3)=5(y^3-z^3),
\]
as subschemes of $Y$. Both $F_1$ and $F_2$ are smooth curves of genus 1 with a rational point $P$, so they are elliptic curves over $\QQ$. One can check computationally that both $F_1$ and $F_2$ have trivial torsion subgroups over $\QQ$. Since $P$ and the divisor corresponding to $\mathcal{O}_{F_i}(M_i)$ are defined over $\QQ$, the point $\alpha_i(P)$ is also rational and, by construction, it is contained in $F_i$. Therefore, $\alpha_i(P)$ is a non-torsion point on $F_i$ for both $i=1,2$, so $P\in Y(\QQ)\backslash Z(\QQ)$, where $Z$ is the closed subset described above. It follows from \cite[Proposition 3.8]{twofibs} that $Y(\QQ)$ is Zariski dense in $Y$.
\end{proof}

\section{Proof of Theorem 1.1} \label{thm1}

\begin{lemma} \label{deg6ext}
Let $X$ be a K3 surface of degree 2 over a number field $K$. Then there is an extension $K'/K$ with $[K': K]\leq 6$ such that $X_{K'}$ contains a singular curve $C$ with geometric genus $p_g(C)=1$. 
\end{lemma}
\begin{proof}
    Since $X$ is a K3 surface of degree 2, it is given by $w^2=f(x,y,z)$ for some homogeneous polynomial $f\in K[x,y,z]$ of degree 6. Consider the double covering
   \begin{align*}
X \ &\xrightarrow{\pi} \ \PP^2_K\\
[x:y:z:w] \ &\mapsto \ [x:y:z],
\end{align*}
and denote the branch curve by $B=V(f)\subset \PP^2_K$. Recall that $B$ is smooth since $X$ is smooth by assumption. 

Let $y_0$ be an arbitrary integer. After potentially changing $y_0$ at most 5 times, we can assume that $f(x,y_0,1)$ is not constant. Thus $f(x,y_0,1)$ is a polynomial in one variable of degree at most 6. If $x_0$ is a root of $f(x,y_0,1)$, we have that $K'=K(x_0)$ has degree $[K':K]\leq 6$. 
By Bézout's theorem, the tangent line $\ell$ to $B$ at $P=[x_0:y_0:1]$ intersects $B$ at 6 points (counting multiplicity). Since $\ell$ is tangent to $B$ at $P$, the intersection of $\ell$ and $B$ at $P$ has multiplicity at least 2. Therefore $\#(B\cap \ell )(\overline{\QQ})\leq 5$. Generically, the tangent line to a plane curve at a point is not tangent at any other points of intersection and has multiplicity 2 at the point of tangency. This guarantees that we can choose $y_0\in \ZZ$ such that the corresponding point $P=[x_0:y_0:1]\in B(K')$ has tangent line $\ell$ satisfying $\#(B\cap \ell)(\overline{\QQ}) = 5$, where $x_0$, $P$ and $K'$ are defined as above for our fixed $y_0$. This means that $\ell$ and $B$ intersect at $P$ with multiplicity 2 and at four other distinct points with multiplicity 1. Note that $\ell=V(ax+by+cz)\subset \PP^2_{K'}$ for some $a,b,c\in K'$ by construction of $K'$. 

We next prove that the (scheme-theoretic) preimage $C\coloneqq \pi^{-1}(\ell)\subset X_{K'}$ is a singular curve with geometric genus 1 defined over $K'$. First note that we can assume after potentially re-choosing $y_0$ that the intersection $(B\cap \ell)(\overline{\QQ})$ does not contain any points with $z$-coordinate zero. Indeed, $B$ only contains finitely many such points, since $\#(B\cap V(z))(\overline{\QQ})\leq 6$ by Bézout's theorem. Then, on the set of lines in $\PP^2$, imposing the conditions that the lines must contain one of the points in $(B\cap V(z))(\overline{\QQ})$ and that they intersect $B$ at a point with multiplicity greater than one gives only finitely many lines, but we had infinitely many choices for $y_0$ and therefore for $\ell$. 
Therefore we can now restrict to the affine chart where $z\neq 0$.

Suppose $b\neq 0$. We have that $C=\pi^{-1}(\ell)\subset X_{K'}$ is given by
\[
w^2=f(x,(-ax-c)/b,1)=\alpha (x-x_0)^2(x-x_1)(x-x_2)(x-x_3)(x-x_4),
\]
for some $\alpha\in K'$, and where $x_0\in K'$ is as above and $x_1,x_2,x_3,x_4$ are the (pairwise different) $x$-coordinates of the four points in $(B\cap \ell)(\overline{\QQ})\backslash \{P\}$. Since $f(x,(-ax-c)/b,1)\in K'[x]$, we have that $C$ is defined over $K'$. One can check that $C$ has a singularity at $[x_0:(-ax_0-c)/b:1:0]$. 

If $b=0$, then $a\neq 0$ since otherwise, $\ell=V(ax+cz)\subset \PP^2_{K'}$ would not contain the point $P=[x_0\colon y_0\colon 1]$. Then, similarly as above, the preimage $C=\pi^{-1}(\ell)\subset X_{K'}$ is given by 
\[
w^2=f(-c/a, y, 1)=(y-y_0)^2h(y),
\]
where $y_0\in K'$ is as above and $h\in K'[y]$ is a separable polynomial of degree 4. In this case, $C$ is singular at $[-c/a:y_0:1:0]$.

In either case, the curve $C$ is of the form $w^2=(x-x_0)^2h(x)$ for a separable polynomial $h\in K'[x]$ of degree 4. Setting $w'=w/(x-x_0)$, we get that $C$ is birationally equivalent to the curve $\widetilde{C}$ given by $(w')^2=h(x)$, which has genus 1. Thus $p_g(C)=p_g(\widetilde{C})=1$, which completes the proof. 
\end{proof}

\begin{proposition} \label{quadraticext}
    Let $C$ be a genus 1 curve over a number field $K$ given by $y^2=h(x)$, where $h\in K[x]$ is a separable polynomial of degree 4. Then there are infinitely many quadratic extensions $L/K$ for which $C_L$ is an elliptic curve with an $L$-rational non-torsion point. 
\end{proposition}
\begin{proof}
    Let $\alpha\in K$ and consider the field extension $K_\alpha=K(\sqrt{h(\alpha)})$. Unless $h(\alpha)=0$, 
the curve $C$ has at least two rational points over $K_\alpha$ corresponding to $w=\pm \sqrt{h(\alpha)}$. 
Let $J$ denote the Jacobian of $C$ over $K$. Since $C$ is smooth and has genus 1, $C(K_\alpha)\neq \emptyset$ implies that $C_{K_\alpha}\cong J_{K_\alpha}$ and therefore $C(K_\alpha)\cong J(K_\alpha)$. 
Letting $\alpha$ vary in $K$, we get infinitely many quadratic extensions $K_{\alpha}$ of $K$ over which $J$ attains a non-identity point. Assume these are all torsion points. Then Merel's theorem (see \cite[p.~242]{AEC} for the statement, \cite{Merel} for the proof) tells us they all have orders bounded above by a fixed integer $N$. But since $J$ is an elliptic curve, we know that its $m$-torsion points $J[m]\coloneqq \{P\in J(\overline{\QQ}) \colon [m]P=O\}$ satisfy \(J[m]\cong \ZZ/m\ZZ \times \ZZ/m\ZZ,\) for any $m\geq 1$. Therefore $J$ cannot have infinitely many distinct points of order at most $N$. 
This implies that, for infinitely many $\alpha\in K$, the curve $C_{K_{\alpha}}$ is an elliptic curve with a non-torsion point. 
\end{proof}

\begin{proof}[Proof of Theorem \ref{bound}]
By Lemma \ref{deg6ext} there is an extension $K'/K$ with $[K' : K]\leq 6$ such that $X_{K'}$ contains a singular curve $C/K'$ of geometric genus 1. By construction of $C$ in the proof of the lemma, we may assume it is given by 
\[
w^2=\alpha (x-x_0)^2(x-x_1)(x-x_2)(x-x_3)(x-x_4),
\]
where $\alpha,x_0\in K'$ and $x_1,x_2,x_3,x_4\in \Bar{\QQ}$ are pairwise distinct and also different from $x_0$. Then, the normalisation $\widetilde{C}$ of $C$ is given by 
\[
w^2=\alpha (x-x_1)(x-x_2)(x-x_3)(x-x_4)\eqcolon h(x),
\]
where $h(x)\in K'[x]$.

By Proposition \ref{quadraticext} there is a quadratic extension $L$ of $K'$ over which $\widetilde{C}_L$ is an elliptic curve with a non-torsion point $Q$. Taking the subgroup generated by $Q$ yields infinitely many points in $\widetilde{C}$ over
$L$,
and since we have the morphisms 
\[
\widetilde{C} \xrightarrow{\nu} C \hookrightarrow X_{K'} \rightarrow X,
\]
we also get infinitely many $L$-rational points in $X$. By construction, $L$ is an extension of $K$ of degree $[L : K]=[L: K'][K': K]\leq 2\cdot 6 =12$.
\end{proof}

\begin{remark}
In the proof of Lemma \ref{deg6ext}, we take the preimage under the double cover $X\rightarrow \PP^2$ of a line tangent to the branch curve and this allows us to get infinitely many rational points on $X$ over a finite field extension. If instead we had taken a transversal line, intersecting the branch curve at 6 distinct points, this method would not have worked: the preimage of such a line would be a curve $C$ of genus 2 by the Riemann--Hurwitz formula, and by Faltings's theorem  \cite{faltings}, $C(K)$ is finite for any number field $K$.
\end{remark}

\begin{remark} \label{rem:constructive}
The proof of Lemma \ref{deg6ext} is constructive and gives a way to find examples of K3 surfaces with infinitely many rational points over their base field, as we now explain. Start with a K3 surface $X$ of degree 2 over a number field $K$ given by $w^2=f(x,y,z)$ such that the branch curve $B=V(f)$ has a rational point. Let $C$ be the smooth curve of genus 1 obtained by normalising the singular curve arising from Lemma \ref{deg6ext}. By construction, $C$ is given by $w^2=h(x)$, where $h\in K[x]$. If $s\coloneqq h(0)\neq 0$, consider the surface $X'$ given by $w^2=f(x,y,z)/s$. Then, repeating the process for $X'$ gives a smooth genus 1 curve $C'$ of the form $w^2=h(x)/s\eqqcolon h'(x)$, so $h'\in K[x]$ satisfies $h'(0)=1$. This means that $C'$ has at least two rational points corresponding to $x=0, w=\pm 1$. Then, $C'$ is an elliptic curve and one can easily determine if one of these points in $C'$ is non-torsion, in which case $X'$ has infinitely many $K$-rational points. This is precisely how the surface $X_0$, used in the proof of Theorem \ref{infpts}, was found.
\end{remark}

\section{K3 surfaces of degree 2 with Picard number 1 and infinitely many rational points} \label{thm2}

In this section we show that there exists a family of K3 surfaces of degree 2 with geometric Picard number 1 and infinitely many rational points over $\QQ$ and we use this to prove Theorem \ref{moduli}. We follow closely the ideas of van Luijk in \cite{vluijk}, where he solved the same problem for quartic K3 surfaces, and we adapt them to K3 surfaces of degree 2. 

Let $h\in\ZZ[x,y,z]$ be a (possibly zero) homogeneous polynomial of degree 6, and let $\mathfrak{X}_h$ be the projective scheme over $\operatorname{Spec}\ZZ$ given by 
\begin{align}\label{X_h}
\begin{split}
73w^2 = & \ 
 7(11x^5y+7x^5z+x^4y^2+5x^4yz+7x^4z^2+7x^3y^3+10x^3y^2z+5x^3yz^2+4x^3z^3 \\
& +6x^2y^4+5x^2y^3z+10x^2y^2z^2+5x^2yz^3+5x^2z^4+11xy^5+5xy^3z^2+12xz^5 \\
& +9y^6+5y^4z^2+10y^2z^4+4z^6+15h(x,y,z)).
\end{split}
\end{align}
We denote its base change to $\QQ$ and $\overline{\QQ}$ by $X_h$ and $\overline{X}_h$, respectively. Note that when $h(x,y,z)=0$, this surface is isomorphic over $\overline{\QQ}$ to the one in \cite[Corollary 30]{eljah}

We will first show that the surfaces $X_h$ are smooth, so they are indeed K3 surfaces. We show it in the same way as van Luijk did in \cite[Theorem 3.1]{vluijk} for a family of quartic K3 surfaces.

\begin{proposition} \label{XhK3}
    Let $h\in \ZZ[x,y,z]$ be a homogeneous polynomial of degree 6. Then the surface $X_h$ is a K3 surface over $\QQ$ of degree 2.
\end{proposition}
\begin{proof}
        Let $X_{\FF_3}/\FF_3$ denote the fibre of $\mathfrak{X}_h \rightarrow \operatorname{Spec} \ZZ$ over $(3)\in \operatorname{Spec} \ZZ$, i.e. $X_{\FF_3}$ arises by reduction modulo 3 of $\mathfrak{X}_h$. 
    Since the factor $15h(x,y,z)$ vanishes when reducing modulo 3, we can directly check that $X_{\FF_3}$ is smooth over $\FF_3$. Consider the base change of $\mathfrak{X}_h$ to the $3$-adic integers $\ZZ_3$, which is projective and flat over $\ZZ_3$. 
    Since $\ZZ_3$ is a discrete valuation ring and the special fibre $X_{\FF_3}\rightarrow \operatorname{Spec} \FF_3$ is smooth, the generic fibre 
    $X_h\rightarrow \operatorname{Spec}\QQ_3$ is also smooth. By the Jacobian criterion, this implies that $X_h\rightarrow \operatorname{Spec} \QQ$ is also smooth.
    Since $X_h$ is a double cover of $\PP^2$ branched along a sextic and we have shown that it is smooth, it is indeed a K3 surface of degree 2.
\end{proof}

\subsection{Computing the geometric Picard number}\label{rho1}

\begin{theorem} \label{piconeh}
     Let $h\in \ZZ[x,y,z]$ be a homogeneous polynomial of degree 6. Then the K3 surface $X_h$ has geometric Picard number 1. 
\end{theorem}
\begin{proof}

Let $\mathcal{Y}_h$ be the projective scheme over $\operatorname{Spec} \ZZ$ given by
 \begin{align*}
\begin{split}
w^2 = & \ 
 11x^5y+7x^5z+x^4y^2+5x^4yz+7x^4z^2+7x^3y^3+10x^3y^2z+5x^3yz^2+4x^3z^3 \\
& +6x^2y^4+5x^2y^3z+10x^2y^2z^2+5x^2yz^3+5x^2z^4+11xy^5+5xy^3z^2+12xz^5 \\
& +9y^6+5y^4z^2+10y^2z^4+4z^6+15h(x,y,z),
\end{split}
\end{align*}
with $Y_h$ and $\overline{Y}_h$ denoting its base changes to $\QQ$ and $\overline{\QQ}$, respectively. Note that over $\overline{\QQ}$ we can change coordinates ($w'=\sqrt{73/7}w)$ to get an isomorphism between $\overline{X}_h$ and $\overline{Y}_h$. Therefore, it suffices to show that $Y_h$ has geometric Picard number 1. This is proved in \cite[Theorem 29 and Corollary 30]{eljah}. 
\end{proof}

\subsection{Finding infinitely many rational points} \label{infsec} 

We now find infinitely many $\QQ$-rational points on a large set of surfaces (this is made precise in Theorem \ref{infpts}) containing the surface $X_0$, defined as in \eqref{X_h} setting $h(x,y,z)=0$. To find infinitely many rational points, we follow the techniques described in the proofs of Lemma \ref{deg6ext} and Theorem \ref{bound}. Some restrictions on the coefficients of the defining equations for the K3 surfaces of degree 2 will guarantee that everything takes place over $\QQ$, so we will not need to enlarge the base field.

There are ${{6+3-1}\choose{3-1}}=28$ monomials of degree 6 in 3 variables and, over $\QQ$, a choice of a degree 6 polynomial $f\in\QQ[x,y,z]$ gives a double cover of $\PP^2$ branched along the curve $V(f)$. Therefore, the double covers of $\PP^2$ branched along a sextic curve given by $w^2=f(x,y,z)$ are parametrised by $M\coloneqq \mathbb{A}^{28}_\QQ$. Let $X\in M$ and let $f$ denote the polynomial defining the branch locus; we define the following properties for $X$:
\begin{itemize}
    \item (sm) X is smooth;
    \item (coeff) $f$ has the coefficients of $x^6, x^5y, x^5z, x^4y^2, x^4yz, x^4z^2$ equal to $0,\frac{77}{73}, \frac{49}{73}, \frac{7}{73}, \frac{35}{73}, \frac{49}{73}$, respectively;
    \item (inf) $X$ has infinitely many rational points;
    \item (eqn) $X$ is defined by an equation as in \eqref{X_h};
    \item (pic1) $X$ is smooth and has geometric Picard number 1.
\end{itemize}

Let $M'\cong \mathbb{A}^{22}_\QQ\subset M$ be the subset of double covers of $\PP^2$ satisfying (coeff). Let $V\subset M$ denote the subset satisfying (sm), i.e. $V$ parametrises the K3 surfaces of degree 2 defined over $\QQ$. Let $W\subset M$ be the subset satisfying (sm) and (coeff), so that $W=V\cap M'$. 

\begin{theorem}\label{infpts}
There is a nonempty Zariski open subset $U\subset M'$ such that every surface $X\in U$ defined over $\QQ$ is a K3 surface of degree 2 with infinitely many $\QQ$-rational points.
\end{theorem}
\begin{proof}
The singular surfaces in $M'$ form a Zariski closed subset of $M'$, so taking its complement we get an open subset $W\subset M'$. By the description of $M'$ above, the surfaces parametrised by $W$ are K3 surfaces of degree 2 and, since the coefficient of the monomial $x^6$ vanishes in their defining polynomial equation, we have $[1:0:0:0]\in X(\QQ)$ for every $X\in W$. Let $X\in W$, with a double cover of $\PP^2$ given by $\pi: X \rightarrow \PP^2$ 
branched along a smooth curve $B_X=V(f)\subset \PP^2$, where $f\in \QQ[x,y,z]$ is a homogeneous polynomial of degree 6. Then, $X$ is given by $w^2=f(x,y,z)$. Since $X\in W \subset M'$, the coefficients of $x^5y$ and $x^5z$ in $f$ are fixed to be $\frac{77}{73}$ and $\frac{49}{73}$, respectively. Then, the tangent line to the branch curve at $[1:0:0]$ is given by
\[
\ell: 11y + 7z = 0.
\]
The equation of this tangent line is the same for every surface in $W$. The curve $C_X\coloneqq \pi^{-1}(\ell)=V(w^2-f, 11y+7z)\subset\PP(1,1,1,3)$ is contained in $X$ and given in the affine chart where $x\neq 0$ by 
\[
w^2=f(1,y,-11/7y).
\] 
In this equation, the constant term and the coefficient of $y$ are zero (by our previous choice of the coefficients of $x^6, x^5y$ and $x^5z$), and since the coefficients of $x^4y^2, x^4z^2$ and $x^4yz$ in $f$ are also fixed to be $\frac{7}{73}, \frac{49}{73}, \frac{35}{73}$, respectively, $y^2$ has coefficient 1. This means that the point $(y,z,w)=(0,0,0)$ or, in projective coordinates, the point $P\coloneqq [1:0:0:0]$, is a double point in $C_X$.  
Generically, $B_X$ and $\ell$ intersect at exactly five points over $\overline{\QQ}$. 
To see this, recall that a polynomial has repeated roots if and only if its discriminant is zero, and this is a closed condition. Since the intersection of $\ell$ with the $B_X$ yields a degree 6 polynomial with a double root at $P$, dividing by the factor corresponding to the double root yields a degree 4 polynomial $g$. The open condition of the polynomial $g$ having no repeated roots is equivalent to $\ell$ intersecting the $B_X$ at $P$ and four other distinct points.

Let $W'$ be the open subset of $W$ corresponding to the surfaces $X$ with branch curve $B_X$ satisfying $\#(B_X\cap \ell)(\overline{\QQ})=5$. Then the normalisation $\widetilde{C}_X$ of $C_X$ is a smooth curve of genus one, as shown in the proof of Theorem \ref{bound}, and it is given by
\[
w^2=f(1,y,-11/7y)/y^2,
\]
where the right-hand side is a degree 4 polynomial in the variable $y$ with constant term 1. Therefore the points $P_1\coloneqq [1:0:0:1]$ and $P_2\coloneqq [1:0:0:-1]$ are rational points in $\widetilde{C}_X$. Specifying a $\QQ$-rational point, say $P_1$, endows $\widetilde{C}_X$ with the structure of an elliptic curve over $\QQ$ with origin $O=P_1$.   

By Mazur's Theorem \cite[Theorem III.5.1]{Mazur}, the point $P_2$ has finite order if and only if $mP_2=O$ for some $m\in \{1,2,\dots,9,10,12\}$. Since $\operatorname{lcm}(1,2,\dots, 9, 10, 12)=2520$, consider the subset of $W'$ of the surfaces $X$ such that $\widetilde{C}_X$ is a genus 1 curve (and therefore an elliptic curve as seen above) satisfying $2520P_2=O$. This is a closed subset, since the equality $2520P_2=O$ is given in coordinates by polynomial equations depending on the coefficients of the Weierstrass equation for $\widetilde{C}_X$. Let $U$ be the complement in $W'$ of this closed subset. Then, for each $X\in U$, the curve $\widetilde{C}_X$ is an elliptic curve with a non-torsion point $P_2$, so $\widetilde{C}_X$ and hence $X$ have infinitely many rational points over $\QQ$.

Finally, we show that $U$ is nonempty. The surface $S=X_0$ defined as in \eqref{X_h} with $h(x,y,z)=0$ is contained in $M'$ as previously observed. It can be checked that it is smooth and that $\widetilde{C}_S$ is a genus one curve given by
\[
\widetilde{C}_S: w^2=\frac{16771780}{1226911}y^4 - \frac{1540220}{175273}y^3 + \frac{81451}{25039}y^2 - \frac{4078}{3577}y + 1.
\]
This is an elliptic curve with a non-identity point over $\QQ$, since it has at least the two rational points $P_1$ and $P_2$ described above. It can be checked computationally (e.g. using \textsc{Magma} \cite{magma}) that the torsion subgroup of $\widetilde{C}_S$ 
is trivial, so the non-identity point we found is a non-torsion point and hence $S\in U$. 
\end{proof}

\subsection{Proof of Theorem \ref{moduli}}
As we saw in Section \ref{infsec}, the double covers of $\PP^2_\QQ$ branched along a sextic curve are parametrised by $M=\mathbb{A}^{28}_\QQ$. Recall that $W\subset M$ is defined by the properties (sm) and (coeff) defined in Section \ref{infsec}. Let $H\subset M$ be defined by the properties (coeff) and (eqn). Proposition \ref{XhK3} shows that (eqn) implies (sm), so $H\subset W$.

\begin{lemma} \label{Hdense}
    The subset $H$ is Zariski dense in $W$.
\end{lemma}
\begin{proof}
 Let $h\in \ZZ[x,y,z]$ be a homogeneous polynomial of degree 6. Note that a surface $X=X_h$ satisfying (eqn) is in $W$ if and only if the coefficients of $x^6$, $x^5y$, $x^5z$, $x^4y^2$, $x^4yz$, $x^4z^2$ vanish in $h\in \ZZ[x,y,z]$. Therefore $H\cong \ZZ^{22}$ as sets. We may view $H$ as a subset of $ \QQ^{22}\cong \mathbb{A}^{22}(\QQ)\cong M'(\QQ)$ 
via the injective map $\varphi: \ZZ^{22}\hookrightarrow\QQ^{22}$
sending the coefficients of the monomials in $h$ to the coefficients of the monomials in $X_h$, in both cases excluding the coefficients of the monomials $x^6$, $x^5y$, $x^5z$, $x^4y^2$, $x^4yz$, $x^4z^2$, which are zero in $h$, and in $X_h$ as in the surfaces parametrised by $M'$. We claim that $H$ is Zariski dense in $M'$. Suppose for a contradiction that $\overline{\varphi(\ZZ^{22})}\neq M'\cong \mathbb{A}_{\QQ}^{22}$. Then $\overline{\varphi(\ZZ^{22})}=V(g)\subset \mathbb{A}_\QQ^{22}$, for some nonzero polynomial $g\in \QQ[x_1,\dots,x_{22}]$. This implies that if $(\alpha_1, \dots, \alpha_{22})\in \ZZ^{22}$, then \[g(\varphi(\alpha_1,\dots,\alpha_{22}))=g(\alpha_1+c_1,\dots,\alpha_{22}+c_{22})=0,\] where $c_i\in \QQ$ are the coefficients of the corresponding monomials in the equation defining $X_0$. 
We may assume without loss of generality that the polynomial 
\(
\widetilde{g}(x)= g(\varphi(x,1,\dots,1))
\)
is nonconstant, so it is a nonzero polynomial in one variable. By our observation above, $\widetilde{g}(\alpha)=0$ for every $\alpha \in \ZZ$, which is not possible.
This proves the claim that $H\subset W\subset M'$ is Zariski dense in $M'$, so in particular, $H$ is Zariski dense in $W$.
\end{proof}

Since $M=\mathbb{A}_\QQ^{28}$ parametrises polynomials of degree 6 in the variables $x,y,z$, we have that $\operatorname{GL}(3, \QQ)$ induces an action on $M$ by acting linearly on $x,y,z$. If there exists $A\in \operatorname{GL}(3, \QQ)$ sending a K3 surface $X\in M(\QQ)$ to $Y\in M(\QQ)$, which we denote by $A(X)=Y$, then $X$ and $Y$ are isomorphic. Over the algebraic closure $\overline{\QQ}$, the double covers of $\PP^2_{\overline{\QQ}}$ branched along a sextic are parametrised, up to isomorphism, by $\PP^{27}_{\overline{\QQ}}$, endowed with the action
 of $\operatorname{PGL}(3, \overline{\QQ})$ (or by $\mathbb{A}^{28}_{\overline{\QQ}}$ with the action of $\operatorname{GL}(3, \overline{\QQ})$).

Let $\mathcal{M}_2$ denote the (coarse) moduli space of K3 surfaces  of degree 2 and consider the projection map
\[
\psi : V \rightarrow \mathcal{M}_2.
\]
Then, $\mathcal{M}_2$ contains an open dense $\mathcal{U}\subset \mathcal{M}_2$ such that, for every $X, Y\in \psi^{-1}(\mathcal{U})$, we have $\psi(X)=\psi(Y)$ if and only if there exists $A\in \operatorname{PGL}(3,\overline{\QQ})$ such that $X=A(Y)$ \cite[\S 5.2]{Huyb}.

\begin{lemma} \label{dominant}
The restriction map $\psi|_{W}$ is dominant. 
\end{lemma}
\begin{proof}
It suffices to check this over the algebraic closure $\overline{\QQ}$. Let $\mathcal{U}\subset \mathcal{M}_2$ be the open dense subset with the property above. Let $X\in \psi^{-1}(\mathcal{U})$ be a K3 surface of degree 2 and let $X\rightarrow \PP^2$ be a double cover branched along a smooth sextic curve $B$. Let $P\in B(\overline{\QQ}) $ and let $\ell$ denote the tangent line to $B$ at $P$. Then there exists a linear transformation $A_1$ sending $P$ to the point $[1:0:0]$ and $\ell$ to the line $V(y)\subset \PP^2$, so $[1:0:0:0]\in A_1(X)(\overline{\QQ})$ and $V(y)$ is tangent to the branch curve of $A_1(X)$ at $[1:0:0]$. This implies that, in the equation defining the K3 surface $A_1(X)$, the coefficients of $x^6$ and $x^5z$ vanish and we can assume that the coefficient of $x^5y$ is $1$. Let $\beta_1, \beta_2, \beta_3 \in \overline{\QQ}$ denote the coefficients of $x^4y^2$, $x^4yz$ and $x^4z^2$ in the equation defining $A_1(X)$, respectively. We claim that there exists a linear transformation $A_2$ sending $A_1(X)$ to a surface such that its defining equation has $0$, $11$, $7$, $1$, $5$ and $7$ as the coefficients of $x^6$, $x^5y$, $x^5z$, $x^4y^2$, $x^4yz$ and $x^4z^2$, respectively.
Once we prove this claim, we are done: we can take the linear transformation $A_3$ sending the variables $x$, $y$ and $z$ to $\sqrt[6]{7/73}x$, $\sqrt[6]{7/73}y$ and $\sqrt[6]{7/73}z$, respectively, and we have that $Y \coloneqq A_3(A_2(A_1(X)))\in V(\overline{\QQ})\cap M'(\overline{\QQ})=W(\overline{\QQ})$, so $\psi|_{W}(Y)=\psi(X)\in \mathcal{U}$. 

To prove the claim, consider a transformation $A_2$ given by
\begin{align*}
x \mapsto &  \ x+ay+bz, \\
y \mapsto & \  11y+7z, \\
z \mapsto & \ cy+dz,
\end{align*}
for some $a,b,c,d\in \overline{\QQ}$. We want to find $a,b,c,d\in \overline{\QQ}$ such that the corresponding transformation $A_2$ is as claimed. Note that, to look at the coefficients of $x^6$, $x^5y$, $x^5z$, $x^4y^2$, $x^4yz$ and $x^4z^2$ in the equation defining the surface $A_2(A_1(X))$, we only need to check how $A_2$ acts on the coefficients of those same six monomials in the equation defining $A_1(X)$, by our choice of transformation $A_2$. Since the coefficient of $x^6$ in the equation of $A_1(X)$ is zero, the coefficient of $x^6$ in the equation of $A_2(A_1(X))$ also vanishes. Looking at how the other coefficients transform under the action of $A_2$, we have that $A_2(A_1(X))$ is as claimed if the following three equations hold.

\begin{align} 
    1 = & \ 121\beta_1+11\beta_2c+\beta_3c^2+55a, \label{x^4y^2} \\
    5 = & \ 154\beta_1+7\beta_2c+11\beta_2d+2\beta_3cd+35a+55b, \label{x^4yz} \\
    7 = & \ 49\beta_1+7\beta_2d+\beta_3d^2+35b. \label{x^4z^2}
\end{align}
Isolating $a$ and $b$ in equations \eqref{x^4y^2} and \eqref{x^4z^2} and substituting them in equation \eqref{x^4yz} yields a polynomial equation in the two variables $c,d$. Set, for example, $c=1$ and choose $d\in \overline{\QQ}$ to be a root of the resulting polynomial in one variable. This gives a transformation $A_2\in \operatorname{GL}(3, \overline{\QQ})$ such that $A_2(A_1(X))$ is as claimed, and this finishes the proof. 
\end{proof}

Let us recall the definitions of the subsets of $M= \mathbb{A}_\QQ^{28}$ involved in the next proof. The subset $M'$ is defined by the property (coeff), $V$ is defined by (sm), $W$ is defined by (sm) and (coeff), $H$ is defined by (coeff) and (eqn) (which implies (sm) by Proposition \ref{XhK3}). Let $U$ denote the nonempty open subset arising from Theorem \ref{infpts}, which satisfies (sm), (coeff) and (inf). Let $T$ be the subset of $U$ also satisfying (eqn), so $T=U\cap H\subset W$. 
Let $S\subset V$ denote the subset satisfying (inf) and (pic1). Then, proving Theorem \ref{moduli} corresponds to showing that $S$ is Zariski dense in the (coarse) moduli space of K3 surfaces of degree 2 denoted by $\mathcal{M}_2$. 
\begin{proof}[Proof of Theorem \ref{moduli}]
Recall that
$ \psi\colon V \rightarrow \mathcal{M}_2$ denotes
the projection map and $W=M'\cap V$.
By Lemma \ref{dominant} we have that $\psi(W)$ is dense in $\mathcal{M}_2$.
Therefore, to show that $\psi(S)$ is Zariski dense in $\mathcal{M}_2$, it suffices to show that $S\cap W$ is dense in $W$. Let $H$, $U$ and $T$ be defined as above. By Lemma \ref{Hdense}, $H$ is Zariski dense in $W$.
As $T\subset U$, we have that $T$ satisfies (inf). Proposition \ref{XhK3} and Theorem \ref{piconeh} imply that $T$ also satisfies (pic1), so $T\subset S\cap W\subset W$. Since $M'$ is irreducible, $U$ is dense in $M'$. The intersection of two dense sets is dense whenever one of them is open, so $T=U\cap H$ is dense in $W$. This implies that $S\cap W$ is also dense in $W$ and this concludes the proof.
\end{proof}


\begin{thebibliography}{99}
{}
\bibitem{Barth}
W. P. Barth, K. Hulek, C. A. M. Peters, and A. Van den Ven. \emph{Compact
Complex Surfaces}. Ergebnisse der Mathematik und ihrer Grenzgebiete. Springer, 2004.
{}
\bibitem{Bog}
F. A. Bogomolov and Yu. Tschinkel. “Density of rational points on elliptic K3 surfaces”.
In: \emph{Asian J. Math.} (2000).
{}
\bibitem{enriques}
F. A. Bogomolov and Yu. Tschinkel. “Density of rational points on Enriques surfaces”.
In: \emph{Math. Res. Letters} (1998), pp. 623–628.
{}
\bibitem{bogo2}
F.A Bogomolov and Yu Tschinkel. “On the density of rational points on elliptic
fibrations”. In: \emph{Journal fur die Reine und Angewandte Mathematik} 511 (1999),
pp. 87–93.
{}

\bibitem{magma}
W.~Bosma, J.~Cannon, and C.~Playoust,
\newblock The Magma algebra system. I. The user language,
\newblock {\em J. Symbolic Comput.}, 24(3--4):235--265, 1997.
\newblock Computational algebra and number theory (London, 1993).

\bibitem{Cornnaka}
P. Corn and M. Nakahara. “Brauer–Manin obstructions on degree 2 K3 surfaces”. In:
\emph{Research in Number Theory} (2018).
{}
\bibitem{eljah}
A. S. Elsenhans and J. Jahnel. “K3 surfaces of Picard rank one and degree two”. In:
\emph{International Algorithmic Number Theory Symposium} (2008), pp. 212–225.
{}
\bibitem{faltings}
G. Faltings. “Endlichkeitssätze für abelsche Varietäten über Zahlkörpern”. In:
\emph{Inventiones mathematicae} 73 (1983), pp. 349–366.
{}
%\bibitem{hurwitz}
%A. Garcia and R. F. Lax. “Rational Nodal Curves with no Smooth Weierstrass Points”.
%In: \emph{Proceedings of the American Mathematical Society} 124.2 (1996), pp. 407–413.
%{}
\bibitem{hasset}
B. Hassett. “Potential Density of Rational Points on Algebraic Varieties”. In:
\emph{Higher Dimensional Varieties and Rational Points}. Vol. 12. Springer Berlin
Heidelberg, Jan. 2003.
{}
\bibitem{Huyb}
D. Huybrechts. \emph{Lectures on K3 surfaces}. Cambridge University Press, 2015.
{}
\bibitem{Liu}
Q. Liu. \emph{Algebraic Geometry and Arithmetic Curves}. Oxford University Press,
2002.
{}
\bibitem{twofibs}
R. van Luijk. “Density of rational points on elliptic surfaces”. In: \emph{Acta
Arithmetica} 156.2 (2012), pp. 189–199.
{}
\bibitem{vluijk}
R. van Luijk. “K3 surfaces with Picard number one and infinitely many rational
points”. In: \emph{Algebra and Number Theory} 1 (July 2005), pp. 1–15.
{}
\bibitem{Mazur}
B. Mazur. “Modular curves and the Eisenstein ideal”. In: \emph{Publications
Mathématiques de l’Institut des Hautes Études Scientifiques} 47 (1977), pp. 33–186.
{}
\bibitem{Merel}
L. Merel. “Bornes pour la torsion des courbes elliptiques sur les corps de nombres”. In:
\emph{Inventiones mathematicae} 124 (1996), pp. 437–449.
{}
\bibitem{mm}
S. Mori and S. Mukai. “The uniruledness of the moduli space of curves of genus 11”. In:
\emph{Lecture Notes in Mathematics}. Vol. 1016. Springer Berlin Heidelberg, Dec. 2006,
pp. 334–353.
{}
\bibitem{AEC}
J.H. Silverman. \emph{The Arithmetic of Elliptic Curves}. Graduate Texts in
Mathematics. Springer, 2009.
\end{thebibliography}
\end{document}